\documentclass[12pt]{amsart}
\usepackage{mathrsfs}

\usepackage{amssymb,amsfonts,amsthm,amsmath}
\usepackage{epsfig}
\usepackage[utf8]{inputenc}
\usepackage{mathrsfs}
\usepackage{lscape}
\usepackage{amssymb,amsmath,graphicx,color,textcomp, amsthm,bbm,bbold, enumerate,booktabs}
\usepackage{chngcntr}
\usepackage{apptools}
\AtAppendix{\counterwithin{theorem}{section}}
\definecolor{Red}{cmyk}{0,1,1,0}

\definecolor{verde}{cmyk}{1,0,1,0}

\definecolor{loka}{cmyk}{.5,0,1,.5}
\definecolor{azul}{cmyk}{1,1,0,0}

\numberwithin{equation}{section}

\newcommand{\be}{\begin{equation}}
\newcommand{\ee}{\end{equation}}

\newtheorem{definition}{Definition}

\newtheorem{teorema}{Theorem}

\begin{document}
\title{On the existence and stability for non-instantaneous impulsive fractional integrodifferential equation}
\author{J. Vanterler da C. Sousa$^1$}
\address{$^1$ Department of Applied Mathematics, Institute of Mathematics,
 Statistics and Scientific Computation, University of Campinas --
UNICAMP, rua S\'ergio Buarque de Holanda 651,
13083--859, Campinas SP, Brazil\newline
Coordination of Civil Engineering, Technological Federal University of Parana, 85053- ´
525, Guarapuava/PR
e-mail: {\itshape \texttt{ra160908@ime.unicamp.br, oliveiradaniela@utfpr.edu.br, capelas@ime.unicamp.br. }}}
\author{D. Santos Oliveira$^2$}
\author{E. Capelas de Oliveira$^1$}

\begin{abstract} In this paper, by means of Banach fixed point theorem, we investigate the existence and Ulam--Hyers--Rassias stability of the non-instantaneous impulsive integrodifferential equation by means of
$\psi$-Hilfer fractional derivative. In this sense, some examples are presented, in order to consolidate the results obtained.

\newline 
%MSC 2010 subject classifications. 26A33, 34A08, 34A12, 34K20, 37C25 .
\end{abstract}
\maketitle

%%%%%%%%%%%%%%%%%%%%%%%%%%%%%%%%%%%%%%%%%%%%%%%%%%%%%%%%%%%%%%%%%%%%%%%%%%%%%%%%%%%%%%%%%%%%%%%%%%%%%%%%%%%%%%%%%%%%%%%%%%%%%%%%%%%%%%%%%%%%%%%%%%%%%%%%%%%%%

\section{Introduction}\label{sec1}

The study of impulsive differential and integrodifferential equations has been
the subject of several papers, in particular, equations involving fractional 
differential and fractional integrodifferential 
operators\cite{exis,exis1,exis2,esto1,esto2,esto3,gronwall}.
However, there are still
few works on the existence, uniqueness and stability of solutions of fractional
differential equations, in particular equations involving the Hilfer fractional
derivative. The many works available usually employ the Caputo
fractional derivative and the Riemann-Liouville fractional derivative
\cite{exis1,exis2,esto3,exis7,exis8,esto8}. With the extension and emergence of
new fractional derivatives of an even number $N$ of variables, one possible way 
to carry new studies of existence, uniqueness, and Ulam--Hyers and 
Ulam--Hyers--Rassias
stabilities, is to use more general fractional derivatives \cite{ZE1}. 

The first ideas about Ulam--Hyers stability began in 1941 with Ulam and Hyers
\cite{rassias,rassias3,rassias4}, from a response by Hyers to a problem
proposed by Ulam, stating the result that is now known as the Ulam--Hyers
stability theorem. In studying stability problems we are trying to answer the
following question: When can we say that the solutions of an inequality are
close to an exact solution of the corresponding equation? Until recently, this
type of question had been raised only for equations involving integer-order
differential operators \cite{rassias3,rassias4}. 

With the expansion of fractional calculus and the increasing number of
researchers working with fractional derivatives and integrals, the
studies about the existence, uniqueness, attractivity and stability of
solutions of fractional differential equations have also grown in number
\cite{exis,exis1,exis2,esto1,esto2,esto3,exis7,exis8,esto8,exis4,exis6,esto5,esto7}.

We can highlight here the investigation by Wang et al. \cite{wang2017}
on a new class of switched systems of differential equations with nonlocal and
impulsive initial conditions, using the Caputo derivative, by means of fixed
point methods. On the other hand, Yang et al. \cite{yang2018}, using the Caputo
fractional derivative, extended some results on the Hausdorff orbital
dependence of solutions of nonlinear differential non-instantaneous equations.
We must also cite the work by Wang et al. \cite{wang2014} in which the authors
investigated the generalized Ulam--Hyers--Rassias and Ulam--Hyers--Rassias
stabilities of a new class of non-instantaneous fractional differential
equation solutions.

Recently, Sousa and Oliveira introduced the so-called $\psi$-Hilfer fractional
derivative, of major importance to fractional calculus as it contains, as
particular cases, a wide class of other fractional derivatives which,
consequently, will have the same properties of the more general fractional
derivative \cite{ZE1}. In this sense, there are already some works on the 
existence, uniqueness and Ulam--Hyers type stability for some classes of equations involving the
$\psi$-Hilfer fractional derivative \cite{ZE3,ZE4,ZE6,ZE7}. 

In this new work we present some results on the existence, uniqueness and
stability of the solutions for a non-instantaneous impulsive fractional integrodifferential
equation in order to consolidate a previous work \cite{ZE1} and to strengthen
the theory of stability within the unified mathematical analysis 
fractional calculus \cite{HER}.

Consider the fractional integrodifferential equation 
\begin{equation}\label{1}
\left\{ 
\begin{array}{ccl}
^{H}\mathbb{D}_{s_{i},0^{+}}^{\alpha ,\beta ,\psi }x\left( t\right)  & = & f\left( t,x(t),\displaystyle\frac{1}{\Gamma (\alpha )}\int_{0}^{t}\mathcal{W}_{\psi }^{\alpha }(t,s,x(s))\,ds\right) ,\text{ }t\in (s_{i},t_{i+1}],\,\,i=0,1,...,m \\ 
x\left( t\right)  & = & g_{i}\left( t,x(t),\displaystyle\frac{1}{\Gamma (\alpha )}\int_{0}^{t}\mathcal{M}_{\psi }^{\alpha }(t,s,x(s))ds\right) ,\,\,\,t\in
(t_{i},s_{i}],i=1,...,m
\end{array}
\right.   
\end{equation}
where ${}^{H}\mathbb{D}_{s_{i},0^{+}}^{\alpha ,\beta ,\psi }(\cdot )$ is the $\psi $-Hilfer fractional derivative, whose parameters satisfy $0<\alpha\leq 1$, $0\leq \beta \leq 1$ \cite{ZE1}.

To simplify the notation we introduce $\mathcal{W}_{\psi }^{\alpha
}(t,s,x(s)):=N_{\psi }^{\alpha }\left( t,s\right) K\left(s,x\left( s\right)
\right) $ and $\mathcal{M}_{\psi }^{\alpha }(t,s,x(s)):=N_{\psi }^{\alpha
}\left( t,s\right) \ell\left( s,x\left( s\right) \right) $ where $N_{\psi
}^{\alpha }(t,s)=\psi'(s)(\psi(t)-\psi(s))^{\alpha-1}$ and
$0=t_{0}=s_{0}<t_{1}\leq s_{1}\leq t_{2}<\cdots <t_{m}\leq s_{m}<t_{m+1}=T$ are
fixed numbers. Also, $f:[0,T]\times \mathbb{R}\times \mathbb{R}\rightarrow
\mathbb{R}$ is continuous; $K,\ell :[0,T]\times \mathbb{R}\rightarrow
\mathbb{R}$ are continuous and $g_{i}:[t_{i},s_{i}]\times \mathbb{R}\times
\mathbb{R}\rightarrow \mathbb{R}$ is continuous for all $i=1,2,\ldots ,m$ which
is called not instantaneous impulse.

When we investigate the existence, uniqueness and Ulam--Hyers type
stability of the solution of a particular integrodifferential equation, we
expect that the results we will obtain will provide contributions to the area.
On the other hand, as it is well known, when we consider non-integer order
derivatives, in some applications, we usually obtain a better description of the
phenomena under study.  Then, the problem of choosing the fractional operator
arises, since the definitions are numerous. Thus, the main motivation for this
work is to present a new class of non-instantaneous impulsive fractional
integrodifferential equations, by means of the  $\psi$-Hilfer fractional derivative
and to investigate the existence and Ulam--Hyers--Rassias type
stability of the solutions of Eq.(\ref{1}), making use of Banach's fixed point
theorem. In this sense, this paper provides new results which are valid for all
possible particular cases of the $\psi$-Hilfer fractional derivative, which is
an advantage of this fractional operator, as the properties of the more general
operator are preserved in the particular cases.

This paper is organized as follows: In section 2 we present some spaces that we
need for the development of the paper. We also present a version of the concept
of Ulam--Hyers--Rassias type stability for the $\psi$-Hilfer fractional
derivative. In section 3, the main results of the paper are presented: the
investigation on the existence and Ulam--Hyers--Rassias
stability of the fractional integrodifferential equation. In section 4, we
present two examples in order to illustrate the results. Concluding
remarks close the paper.

%%%%%%%%%%%%%%%%%%%%%%%%%%%%%%%%%%%%%%%%%%%%%%%%%%%%%%%%%%%%%%%%%%%%%%%%%%%%%%%%%%%%%%%%%%%%%%%%%%%%%%%%%%%%%%%%%%%%%%%%%%%%%%%%%%%%%%%%%%%%%%%%%%%%%%%%%%%%%%%%%%%%%%%%%%%%%%%%%%%%%%%%%%%%%%%%%%%%%%%%%%%%%%%%%%%%%%%%%%%%%%%%

\section{Preliminaries}

Let $J=[0,T]$, $J^{\prime }=(0,T]$ and $C(J,\mathbb{R})$ be the space of all continuous functions from $J$ into $\mathbb{R}$ and $n$-times absolutely continuous set, given by \cite{ZE1}
\begin{equation*}
AC^{n}(J,\mathbb{R})=\left\{ x:J\rightarrow \mathbb{R},x^{(n-1)}\in AC(J,\mathbb{R})\right\}.
\end{equation*}

The weighted space of functions on $J^{\prime }$ is
\begin{equation*}
C_{\gamma ,\psi }(J,\mathbb{R})=\left\{ x:J^{\prime }\rightarrow \mathbb{R},(\psi (t)-\psi (a))^{\gamma }x(t)\in C(J,\mathbb{R})\right\} ,\quad 0\leq \gamma <1.
\end{equation*}

The weighted space of functions on $J^{\prime }$ is
\begin{equation*}
C_{\gamma ,\psi }^{n}(J,\mathbb{R})=\left\{ x:J^{\prime }\rightarrow \mathbb{R},x(t)\in C^{n-1}(J,\mathbb{R}),\text{ }x^{(n)}(t)\in C_{\gamma ,\psi }(J,\mathbb{R})\right\} ,\quad 0\leq \gamma <1.
\end{equation*}

The weighted space is:
\begin{equation*}
C_{\gamma ,\psi }^{\alpha ,\beta }(J,\mathbb{R})=\left\{ x\in C_{\gamma ,\psi }(J,\mathbb{R}),{}^{H}\mathbb{D}_{a^{+}}^{\alpha ,\beta ,\psi }x\in C_{\gamma ,\psi }(J,\mathbb{R})\right\} ,\quad \gamma =\alpha +\beta (1-\alpha ).
\end{equation*}

We also recall the piecewise weighted function space 
\begin{equation*}
PC_{\gamma ,\psi }(J,\mathbb{R})=\left\{ (\psi (t)-\psi (a))^{\gamma}x(t)\in C_{\gamma ,\psi }((t_{k},t_{k+1}],\mathbb{R}),\,\,k\in \lbrack 0,m]\right\} 
\end{equation*}
and there exist $x(t_{k}^{-})$ and $x(t_{k}^{+})$, $k=1,2,\ldots ,m$ with $x(t_{k}^{-})=x(t_{k}^{+})$.

Meanwhile, we set 
\begin{equation*}
PC_{\gamma ,\psi }^{1}(J,\mathbb{R})=\left\{ x\in PC_{\gamma ,\psi }(J,\mathbb{R}),\,\,x^{\prime }\in PC_{\gamma ,\psi }(J,\mathbb{R})\right\} 
\end{equation*}
and 
\begin{equation*}
PC_{\gamma ,\psi }^{\alpha ,\beta }(J,\mathbb{R})=\left\{ x\in PC_{\gamma,\psi }(J,\mathbb{R}),\,\,{}^{H}\mathbb{D}_{0^{+}}^{\alpha ,\beta ,\psi
}x\in PC_{\gamma ,\psi }(J,\mathbb{R})\right\}.
\end{equation*}

The function $x\in PC_{\gamma ,\psi }(J,\mathbb{R})$ is called a classical solution of the fractional impulsive Cauchy problem
\begin{equation}
\left\{ 
\begin{array}{ccl}
^{H}\mathbb{D}_{s_{i},0^{+}}^{\alpha ,\beta ,\psi }x(t) & = & f\left( t,x(t),\displaystyle\frac{1}{\Gamma (\alpha )}\int_{0}^{t}\mathcal{W}_{\psi }^{\alpha
}(t,s,x(s))\,ds\right) ,\,\,t\in (s_{i},t_{i+1}],\text{ }i=0,1,...,m \\ 
x(t) & = & g_{i}\left( t,x(t),\displaystyle\frac{1}{\Gamma (\alpha )}\int_{0}^{t}\mathcal{M}_{\psi }^{\alpha }(t,s,x(s))ds\right) ,\,\,t\in (t_{i},s_{i}],%
\text{ }i=1,...,m \\ 
I_{0^{+}}^{1-\gamma ,\psi }x(0) & = & x_{0}
\end{array}
\right. 
\end{equation}
if $x$ satisfies $I_{0^{+}}^{1-\gamma ,\psi }x(0)=x_{0}$ and 
\begin{equation*}
x(t)=g_{i}\left( t,x(t),\frac{1}{\Gamma (\alpha )}\int_{0}^{t}\mathcal{M}_{\psi}^{\alpha }(t,s,x(s))ds\right) ,\,\,t\in (t_{i},s_{i}],\,\,i=0,1,...,m
\end{equation*}
and
\begin{equation*}
\left\{ 
\begin{array}{cll}
x(t) & = & \Psi ^{\gamma }\left( t,0\right) x_{0}+\displaystyle\frac{1}{\Gamma (\alpha )}\int_{0}^{t}N_{\psi }^{\alpha }(t,s)f\left( s,x(s),\displaystyle\frac{1}{\Gamma (\alpha )}\int_{0}^{s}\mathcal{W}_{\psi }^{\alpha}(s,\tau ,x(\tau ))d\tau \right) ds,\,\,t\in \lbrack 0,t_{1}] \\ 
x(t) & = & 
\begin{array}{l}
\left. \displaystyle\frac{1}{\Gamma (\alpha )}\int_{s_{i}}^{t}N_{\psi}^{\alpha }(t,s)f\left( s,x(s),\displaystyle\frac{1}{\Gamma (\alpha )}\int_{0}^{s}\mathcal{W}_{\psi }^{\alpha }(s,\tau ,x(\tau ))d\tau \right) ds +g_{i}\left( s_{i},x(s_{i}),\displaystyle\frac{1}{\Gamma (\alpha )}\int_{0}^{s_{i}}\mathcal{M}_{\psi }^{\alpha }(t,s,x(s))ds\right) \right.,  \\ \left.t\in (s_{i},t_{i+1}]\,,i=0,1,...,m\right. 
\end{array}
\end{array}
\right. 
\end{equation*}
with $\Psi ^{\gamma }\left( t,0\right) =\displaystyle\frac{\left( \psi \left( t\right) -\psi \left( 0\right) \right) ^{\gamma -1}}{\Gamma \left( \gamma \right) }$ and $N_{\psi }^{\alpha }(t,s)=\psi ^{\prime }\left( s\right)\left( \psi \left( t\right) -\psi \left( s\right) \right) ^{\alpha -1}$.

Then, we introduce the concept of Ulam--Hyers--Rassias stability for Eq.(\ref{1}), motivated by the concepts of stability \cite{princi,princi1,princi2}.

Let $\delta \geq 0$ and $\varphi \in PC_{\gamma,\psi}(J,\mathbb{R}_{+})$ \textbf{be} a nondecreasing function. Consider
\begin{equation}
\left\vert {}^{H}\mathbb{D}_{s_{i},0^{+}}^{\alpha ,\beta ,\psi }y(t)-f\left(t,y(t),\frac{1}{\Gamma (\alpha )}\int_{0}^{t}\mathcal{W}_{\psi }^{\alpha
}(t,s,y(s))ds\right) \right\vert \leq \varphi (t),\,\,t\in (s_{i},t_{i+1}],\,i=0,1,...,m, \label{3}
\end{equation}
and 
\begin{equation}
\left\vert y(t)-g_{i}\left( t,y(t),\dfrac{1}{\Gamma (\alpha )}\int_{0}^{t}\mathcal{M}_{\psi }^{\alpha }(t,s,y(s))ds\right) \right\vert \leq \varepsilon,\,\,t\in (t_{i},s_{i}]\,,\,i=1,...,m.  \label{3.1}
\end{equation}

\begin{definition} {\rm \cite{princi2}} The {\rm{Eq.(\ref{1})}} is generalized Ulam--Hyers--Rassias stable with respect to $(\varphi ,\delta )$ if there exists $C_{f,g_{i},\varphi}>0$ such that for each solution $y\in PC_{\gamma ,\psi }^{1}(J,\mathbb{R})$ of the {\rm{Eq.(\ref{1})}} we have
\begin{equation}
|y(t)-x(t)|\leq C_{f,g_{i},\varphi }(\varphi (t)+\delta ),\,\,t\in J.
\end{equation}

Obviously, if $y\in PC_{\gamma ,\psi }^{1}(J,\mathbb{R})$ satisfies the inequality {\rm{Eq.(\ref{3})}} then $y$ is a solution the following fractional integral inequality 
\begin{equation}\label{4}
\left\vert y(t)-g_{i}\left( t,y(t),\frac{1}{\Gamma (\alpha )}\int_{0}^{t}\mathcal{M}_{\psi }^{\alpha }(t,s,y(s))ds\right) \right\vert \leq \delta,\,\,t\in (t_{i},s_{i}],\,\,i=0,1,...,m  
\end{equation}%
\begin{equation}\label{5}
\left\vert y(t)-\Psi ^{\gamma }\left( t,0\right) x_{0}-\int_{0}^{t}N_{\psi}^{\alpha }(t,s)f\left( s,y(s),\frac{1}{\Gamma (\alpha )}\int_{0}^{s}\mathcal{W}_{\psi
}^{\alpha }(s,\tau ,y(\tau ))d\tau \right) ds\right\vert \leq I_{0^{+}}^{\alpha ,\psi }\varphi  
\end{equation}%
$t\in \lbrack 0,t_{1}]$ and 
\begin{eqnarray}
y(t)-g_{i}\left( s_{i},y(s_{i}),\frac{1}{\Gamma (\alpha )}\int_{0}^{s_{i}}\mathcal{M}_{\psi }^{\alpha }(t,s,y(s))ds\right)-\frac{1}{\Gamma (\alpha )}\int_{s_{i}}^{t}N_{\psi }^{\alpha }(t,s)f\left(s,y(s),\frac{1}{\Gamma (\alpha )}\int_{0}^{s}\mathcal{W}_{\psi }^{\alpha }(s,\tau
,y(\tau ))d\tau \right) ds &\leq &\delta +I_{s_{i}}^{\alpha ,\psi }\varphi (t) \notag \\
\end{eqnarray}%
$t\in (s_{i},t_{i+1}]$ with $i=1,...,m$.
\end{definition}

For a nonempty set $X$ a function ${\mbox{d}}: X \times X \to [0,\infty]$ is called a metric on $X$ if, and only if, ${\mbox{d}}$ satisfies \cite{ZE3,ZE4}:
\begin{enumerate}
\item ${\mbox{d}}(x,y)=0$ if and only if $x=y$;
\item ${\mbox{d}}(x,y)={\mbox{d}}(y,x)$ for all $x,y \in X$;
\item ${\mbox{d}}(x,z) \leq {\mbox{d}}(x,y) + {\mbox{d}}(y,z)$ for all $x,y,z \in X$
\end{enumerate}

As the Banach fixed point theorem is important in the study of existence, uniqueness, and in the stability of solution of fractional differential equations, it is here after recalled.

\begin{teorema}\label{teo1}{\rm \cite{ZE3,ZE4}} Let $(X,{\mbox{d}})$ be a generalized complete metric space. Assume that $T: X \to X$ is a strictly contractive with constant $L < 1$. If there exists a nonnegative integer $k$ such that ${\mbox{d}}(T^{k+1} x, T^k x) < \infty$ for any $x \in  X$, then the followings are true:
\begin{enumerate}
\item The sequence $\{T^n x\}$ converges to a fixed point $x^{*}$ of $T$;
\item $x^{*}$ is the unique fixed point of $T$ in $X^{*} = \{ y \in X | {\mbox{d}}(T^{*}x,y) < \infty \}$;
\item If $y \in X^{*}$, then ${\mbox{d}}(y,x^{*}) \leq \dfrac{1}{1-L} {\mbox{d}}(Ty,y)$.
\end{enumerate}
\end{teorema}

%%%%%%%%%%%%%%%%%%%%%%%%%%%%%%%%%%%%%%%%%%%%%%%%%%%%%%%%%%%%%%%%%%%%%%%%%%%%%%%%%%%%%%%%%%%%%%%%%%%%%%%%%%%%%%%%%%%%%%%%%%%%%%%%%%%%%%%%%%%%%%%%%%%%%%%%%%%%%%%%%%%%%%%%%%%%%%%%%%%%%%%%%%%%%%%%%%%%%%%%%%%%%%%%%%%%%%%%%%%%%%%%

\section{Main results}
In this section, we investigate the existence and Ulam--Hyers--Rassias stability of the solution of Eq.(\ref{1}) by using Banach's fixed point theorem.

For the development of this paper, we first introduce some hypotheses

\begin{tabular}{cl}
$H_1.$ & $f \in C(J\times \mathbb{R}\times \mathbb{R}, \mathbb{R})$\\
$H_2.$ & There exists a positive constant $L_f$ such that 
\end{tabular}
\begin{equation*}
|f(t,u_1,v_1) - f(t,u_2,v_2)| \leq L_f (|u_1-u_2|+ |v_1-v_2|)
\end{equation*}
for each $t \in J$ and for all $u_1,u_2,v_1,v_2 \in \mathbb{R}$.

\begin{tabular}{cl}
$H_3.$ & $g_{i}\in C([t_i,s_i] \times \mathbb{R} \times \mathbb{R}, \mathbb{R})$ and $L_{g_{i}}>0$ constants,\\ 
& $i=1,2,\ldots,m$ such that 
\end{tabular}
\begin{equation*}
|g_i(t,u_1,v_1) - g_i(t,u_2,v_2)| \leq L_{g_i} (|u_1-u_2|+ |v_1-v_2|),
\end{equation*}
for each $t\in(t_{i},s_{i}]$ and for all $u_1,u_2,v_1,v_2 \in \mathbb{R}$. 

\begin{tabular}{cl}
$H_4.$ & $k \in C([t_i,s_i] \times \mathbb{R} \times \mathbb{R}, \mathbb{R})$ and $\overline{K}>0$ constant, such that 
\end{tabular}
\begin{equation*}
|k(t,u_1) - k(t,u_2)| \leq \overline{K} |u_1-u_2|
\end{equation*}
for each $t \in J$ and for all $u_1,u_2 \in \mathbb{R}$. 

\begin{tabular}{cl}
$H_5.$ & $\ell \in C(J \times \mathbb{R}, \mathbb{R})$ and $L>0$ constant, such that 
\end{tabular}
\begin{equation*}
|\ell(t,u_1) - \ell(t,u_2)| \leq L |u_1-u_2|
\end{equation*}
for each $t \in J$ and for all $u_1,u_2 \in \mathbb{R}$. 

\begin{tabular}{cl}
$H_6.$ & Let $\varphi \in C(J, \mathbb{R})$ be \textbf{a} nondecreasing function. There exists $C_{\varphi} > 0$ such that 
\end{tabular}
\begin{equation*}
I_{0^{+}}^{\alpha,\psi} \varphi (t) := \frac{1}{\Gamma(\alpha)} \int_0^t N_{\psi}^{\alpha}(t,s) \varphi(s) ds \leq C_{\varphi} \varphi(t)
\end{equation*}
for each $t\in J$.

In what follows we will investigate the existence and generalized stability of Ulam--Hyers--Rassias of the solutions of Eq.(\ref{1}) by means of Banach fixed point theorem (Theorem \ref{teo1}).

\begin{teorema}\label{teo2.1} Assume that the hypotheses $H_{1},H_{2},H_{3},H_{4},H_{5},H_{6}$ are satisfied and a function $y\in PC_{\gamma ,\psi }^{1}(J,\mathbb{R})$, satisfies {\rm{Eq.(\ref{3})}}. Then, there exists a unique solution $y_{0}:J\rightarrow \mathbb{R}
$ such that 
\begin{equation*}
y_{0}(t)=\Psi ^{\gamma }\left( t,0\right) x(0)+\frac{1}{\Gamma (\alpha )}\int_{0}^{t}N_{\psi }^{\alpha }(t,s)f\left( s,y_{0}(s),\frac{1}{\Gamma(\alpha )}\int_{0}^{s}\mathcal{W}_{\psi }^{\alpha }(s,\tau ,y_{0}(\tau ))d\tau \right)ds
\end{equation*}
for $t\in \lbrack 0,t_{1}]$; and 
\begin{equation*}
y_{0}(t)=g_{i}\left( t,y_{0}(t),\frac{1}{\Gamma \left( \alpha \right) }\int_{0}^{t}\mathcal{M}_{\psi }^{\alpha }(t,s,y_{0}(s))ds\right) 
\end{equation*}%
for $t\in (t_{i},s_{i}]$ and $i=1,...,m$, and 
\begin{eqnarray}\label{2.1}
y_{0}(t) &=&g_{i}\left( s_{i},y_{0}(s_{i}),\frac{1}{\Gamma \left( \alpha \right) }\int_{0}^{s_{i}}\mathcal{M}_{\psi }^{\alpha }(t,s,y_{0}(s))ds\right)+ \notag \\
&&\frac{1}{\Gamma (\alpha )}\int_{s_{i}}^{t}N_{\psi }^{\alpha }(t,s)f\left( s,y_{0}(s),\frac{1}{\Gamma (\alpha )}\int_{0}^{s}\mathcal{W}_{\psi }^{\alpha }(s,\tau ,y_{0}(\tau ))d\tau \right) ds 
\end{eqnarray}%
for $t\in (s_{i},t_{i+1}]$, $i=1,...,m$ and 
\begin{equation}
\left\vert y(t)-y_{0}(t)\right\vert \leq \frac{(1+C_{\varphi })(\varphi	(t)+\delta )}{1-\Phi },\,\,t\in J  \label{2.2}
\end{equation}
provided that
\begin{equation} \label{2.3}
\Phi :=\underset{i=1,...,m}{\max }\left\{ 
\begin{array}{c}
\left( LC_{\varphi }+L\dfrac{(\psi (T)-\psi (0))^{\alpha }}{\Gamma (\alpha +1)}+1\right) L_{g_{i}}+\left( \overline{K}\dfrac{(\psi (T)-\psi (0))^{2\alpha }}{\Gamma (\alpha +1)}C_{\varphi }^{2}+\overline{K}C_{\varphi }^{2}+C_{\varphi }\right) L_{f}
\end{array}
\right\} <1.
\end{equation}
\end{teorema}

\begin{proof}

The proof of the theorem will be carried out in three cases. First, we consider the space of piecewise weighted functions 
\begin{equation}
X:=\left\{ g:J\rightarrow \mathbb{R}/g\in PC_{\gamma ,\psi }(J,\mathbb{R})\right\}   \label{2,4}
\end{equation}%
and the generalized metric on $X$, defined by 
\begin{equation}
{\mbox{d}}(g,h)=\mathrm{Inf}\left\{ C_{1}+C_{2}\in \lbrack 0,\infty ]/|g(t)-h(t)|\leq (C_{1}+C_{2})(\varphi (t)+\delta ),\,\,t\in J\right\} \label{2.5}
\end{equation}
where 
\begin{equation*}
C_{1}\in \left\{ C\in \lbrack 0,\infty ]/|g(t)-h(t)|\leq C\varphi(t),\,\,t\in (s_{i},t_{i+1}],\,\,i=0,1,...,m\right\} 
\end{equation*}%
and 
\begin{equation*}
C_{2}\in \left\{ C\in \lbrack 0,\infty ]/|g(t)-h(t)|\leq C\delta ,\,\,t\in (t_{i},s_{i}],\,\,i=1,2,...,m\right\}.
\end{equation*}

Note that, the $(X,{\mbox{d}})$ is a complete generalized metric space. 

Also, we introduce the following operator $\Omega :X\rightarrow X$ given by 
\begin{equation*}
(\Omega x)(t)=\Psi ^{\gamma }\left( t,0\right) x(0)+\frac{1}{\Gamma (\alpha )}\int_{0}^{t}N_{\psi }^{\alpha }(t,s)f\left( s,x(s),\frac{1}{\Gamma (\alpha )}\int_{0}^{s}\mathcal{W}_{\psi }^{\alpha }(s,\tau ,x(\tau ))d\tau \right) ds, \text{ } \mbox{for $t\in \lbrack 0,t_{1}]$}
\end{equation*}
and 
\begin{equation*}
(\Omega x)(t)=g_{i}\left( t,x(t),\frac{1}{\Gamma (\alpha }\int_{0}^{t}\mathcal{M}_{\psi }^{\alpha }(t,s,x(s))ds\right) ,\text{ } \mbox{for $t\in (t_{i},s_{i}]$ and $\,i=0,1,...,m$;}
\end{equation*}%
and 
\begin{eqnarray}\label{eq432}
(\Omega x)(t) &=&g_{i}\left( s_{i},x(s_{i}),\frac{1}{\Gamma (\alpha )}\int_{0}^{s_{i}}\mathcal{M}_{\psi }^{\alpha }(t,s,x(s))ds\right)+ \notag \\
&&\frac{1}{\Gamma (\alpha )}\int_{s_{i}}^{t}N_{\psi }^{\alpha }(t,s)f\left(s,x(s),\frac{1}{\Gamma (\alpha )}\int_{0}^{s}\mathcal{W}_{\psi }^{\alpha }(s,\tau ,x(\tau ))d\tau \right) ds 
\end{eqnarray}
for $t\in (s_{i},t_{i+1}]$ and $\,i=0,1,...,m$, for all $x\in X$ and $t\in \lbrack 0,T]$. Note that, the operator $\Omega $, as defined above, is a well defined operator according to $H_{1},H_{3},H_{4}$ and $H_{5}$.

The definition of metric ${\rm d}$ over the space $X$ for any $g,h\in X$, \textbf{allows} to find $C_{1},C_{2}\in \lbrack 0,\infty ]$ such that 
\begin{equation} \label{2.7}
|g(t)-h(t)|\leq \left\{ 
\begin{array}{cc}
C_{1}\varphi \left( t\right)  & t\in (s_{i},t_{i+1}),\,\,i=0,1,...,m, \\ 
C_{2}\delta  & t\in (t_{0},s_{i}),\,\,i=1,...,m.%
\end{array}
\right. 
\end{equation}

By means of the definition of $\Omega $ in {\rm{Eq.(\ref{eq432})}}, $H_{2},H_{3},H_{4},H_{5}$ and {\rm{Eq.(\ref{2.7})}} we have the following cases:

\noindent\textsf{Case 1.} $t \in [0,t_1]$.

\textbf{We} have 
\begin{eqnarray*}
\left\vert (\Omega g)(t)-(\Omega h)(t)\right\vert  &=&\left\vert \frac{1}{\Gamma (\alpha )}\int_{0}^{t}N_{\psi }^{\alpha }(t,s)f\left( s,g(s),\frac{1}{\Gamma (\alpha )}\int_{0}^{s}\mathcal{W}_{\psi }^{\alpha }(s,\tau ,g(\tau ))d\tau\right) - f\left( s,h(s),\frac{1}{\Gamma (\alpha )}\int_{0}^{s}\mathcal{W}_{\psi}^{\alpha }(s,\tau ,h(\tau ))d\tau \right) ds\right\vert ;
\end{eqnarray*}
so,
\begin{eqnarray*}
&&\left\vert (\Omega g)(t)-(\Omega h)(t)\right\vert  \notag \\
&\leq &\frac{1}{\Gamma (\alpha )}\int_{0}^{t}N_{\psi }^{\alpha }(t,s)L_{f}\left\{ |g(s)-h(s)|+\frac{1}{\Gamma (\alpha )}\int_{0}^{s}N_{\psi }^{\alpha }(s,\tau )\left\vert k(\tau ,g(\tau )-k(\tau ,h(\tau ))\right\vert d\tau \right\} ds  \notag \\
&\leq &\frac{1}{\Gamma (\alpha )}\int_{0}^{t}N_{\psi }^{\alpha }(t,s)L_{f}\left\{ C_{1}\varphi (s)+\frac{1}{\Gamma (\alpha )} \int_{0}^{s}N_{\psi }^{\alpha }(s,\tau )\overline{K}\left\vert g(\tau )-h(\tau )\right\vert d\tau \right\} ds  \notag \\
&\leq &\frac{1}{\Gamma (\alpha )}\int_{0}^{t}N_{\psi }^{\alpha }(t,s)L_{f}\left\{ C_{1}\varphi (s)+\frac{C_{1}\overline{K}}{\Gamma (\alpha ) }\int_{0}^{s}N_{\psi }^{\alpha }(s,\tau )\varphi (\tau )d\tau \right\} ds \notag \\
&\leq &\frac{1}{\Gamma (\alpha )}\int_{0}^{t}N_{\psi }^{\alpha }(t,s)L_{f}\left\{ C_{1}\varphi (s)+C_{1}\overline{K}C_{\varphi }\varphi (s)\right\} ds
\end{eqnarray*}
or, 
\begin{eqnarray*}
\left\vert (\Omega g)(t)-(\Omega h)(t)\right\vert  &\leq & \frac{L_{f}C_{1}(1+\overline{K}C_{\varphi})}{\Gamma(\alpha)}   \int_{0}^{t}N_{\psi }^{\alpha}(t,s)\varphi (s)ds \leq L_{f}C_{1}(1+\overline{K}C_{\varphi })C_{\varphi }\varphi (t).
\end{eqnarray*}

Therefore, we have 
\begin{equation*}
|(\Omega g)(t)-(\Omega h)(t)|\leq L_{f}C_{1}(1+\overline{K}C_{\varphi})C_{\varphi }\varphi (t).
\end{equation*}

\noindent {\sf Case 2.} $t \in (t_i,s_i]$.

it holds
\begin{equation*}
\left\vert (\Omega g)(t)-(\Omega h)(t)\right\vert =\left\vert 
\begin{array}{c}
g_{i}\left( t,g(t),\displaystyle\frac{1}{\Gamma (\alpha )}\int_{0}^{t}\mathcal{M}_{\psi }^{\alpha }(t,s,g(s))ds\right) -g_{i}\left( t,h(t),\displaystyle\frac{1}{\Gamma (\alpha )}\int_{0}^{t}\mathcal{M}_{\psi }^{\alpha }(t,s,h(s))ds\right) 
\end{array}
\right\vert ,
\end{equation*}
so, 
\begin{eqnarray*}
\left\vert (\Omega g)(t)-(\Omega h)(t)\right\vert &\leq &L_{g_{i}}\left(|g(t)-h(t)|+\frac{1}{\Gamma (\alpha )}\int_{0}^{t}N_{\psi }^{\alpha
}(t,s)|\ell (s,g(s))-\ell (s,h(s))|ds\right)   \notag \\
&\leq &L_{g_{i}}\left( C_{2}\delta +\frac{1}{\Gamma (\alpha )}\int_{0}^{t}N_{\psi }^{\alpha }(t,s)L(C_{1}+C_{2})(\varphi (s)+\delta)ds\right)   \notag \\
&=&L_{g_{i}}\left[ C_{2}\delta +L(C_{1}+C_{2})\left\{ \frac{1}{\Gamma (\alpha )}\int_{0}^{t}N_{\psi }^{\alpha }(t,s)\varphi (s)ds+\frac{1}{\Gamma
(\alpha )}\int_{0}^{t}N_{\psi }^{\alpha }(t,s)\delta ds\right\} \right]  \notag \\
&\leq &L_{g_{i}}\left[ C_{2}\delta +L(C_{1}+C_{2})\left( C_{\varphi }\varphi (t)+\delta \frac{(\psi (T)-\psi (0))^{\alpha }}{\Gamma (\alpha +1)}\right)\right] .
\end{eqnarray*}

Therefore, we have
\begin{equation*}
|(\Omega g)(t)-(\Omega h)(t)|\leq L_{g_{i}}\left[ C_{2}\delta+L(C_{1}+C_{2})\left( C_{\varphi }\varphi (t)+\delta \frac{(\psi (T)-\psi(0))^{\alpha }}{\Gamma (\alpha +1)}\right) \right].
\end{equation*}

\noindent {\sf Case 3.} For $t \in (s_i,t_{i+1}]$, we have
\begin{eqnarray*}
&&\left\vert (\Omega g)(t)-(\Omega h)(t)\right\vert  \notag \\
&=&\left\vert g_{i}\left( s_{i},g(s_{i}),\frac{1}{\Gamma (\alpha )}\int_{0}^{s_{i}}\mathcal{W}_{\psi }^{\alpha }(t,s,g(s))ds\right) + \frac{1}{\Gamma (\alpha )}\int_{s_{i}}^{t}N_{\psi }^{\alpha }(t,s)f\left( s,g(s),\frac{1}{\Gamma (\alpha )}\int_{0}^{s}\mathcal{W}_{\psi }^{\alpha
}(s,\tau ,g(\tau ))d\tau \right) ds\right\vert  \\
&&-\left\vert g_{i}\left( s_{i},h(s_{i}),\frac{1}{\Gamma (\alpha )}\int_{0}^{s_{i}}\mathcal{M}_{\psi }^{\alpha }(t,s,h(s))ds\right) -\frac{1}{\Gamma (\alpha )}\int_{s_{i}}^{t}N_{\psi }^{\alpha}(t,s)f\left( s,h(s),\frac{1}{\Gamma (\alpha )}\int_{0}^{s}\mathcal{W}_{\psi }^{\alpha
}(s,\tau ,h(\tau ))d\tau \right) ds\right\vert; 
\end{eqnarray*}
\textbf{therefore,}
\begin{eqnarray*}
&&\left\vert (\Omega g)(t)-(\Omega h)(t)\right\vert   \notag \\
&\leq &L_{g_{i}}\left\{ |g(s_{i})-h(s_{i})|+\frac{1}{\Gamma (\alpha )}\int_{0}^{s_{i}}N_{\psi }^{\alpha }(t,s)|\ell (s,g(s))-\ell
(s,h(s))|ds\right\} +  \notag \\
&&\frac{1}{\Gamma (\alpha )}\int_{s_{i}}^{t}N_{\psi }^{\alpha}(t,s)L_{f}\left\{ |g(s)-h(s)|+\frac{1}{\Gamma (\alpha )}\int_{0}^{s}N_{\psi}^{\alpha }(s,\tau )|k(\tau ,g(\tau ))-k(\tau ,h(\tau ))|d\tau \right\} ds \notag \\
&\leq &L_{g_{i}}\left\{ C_{2}\delta +\frac{1}{\Gamma (\alpha )}\int_{0}^{s_{i}}N_{\psi }^{\alpha }(t,s)L|g(s)-h(s)|ds\right\} +  \notag \\
&&\frac{1}{\Gamma (\alpha )}\int_{s_{i}}^{t}N_{\psi }^{\alpha}(t,s)L_{f}\left\{ C_{1}\varphi (s)+\frac{1}{\Gamma (\alpha )}\int_{0}^{s}N_{\psi }^{\alpha }(s,\tau )\overline{K}|g(\tau )-h(\tau )|d\tau\right\} ds  \notag \\
&\leq &L_{g_{i}}\left\{ C_{2}\delta +\frac{L}{\Gamma (\alpha )}\int_{0}^{s_{i}}N_{\psi }^{\alpha }(t,s)(C_{1}+C_{2})(\varphi (s)+\delta)ds\right\} +  \notag \\
&&\frac{L_{f}}{\Gamma (\alpha )}\int_{s_{i}}^{t}N_{\psi }^{\alpha}(t,s)L_{f}\left\{ C_{1}\varphi (s)+\frac{\overline{K}}{\Gamma (\alpha )}\int_{0}^{s}N_{\psi }^{\alpha }(s,\tau )(C_{1}+C_{2})(\varphi (\tau )+\delta)d\tau \right\} ds  \notag \\
&\leq &L_{g_{i}}\left\{ C_{2}\delta +L(C_{1}+C_{2})\left( C_{\varphi}\varphi (t)+\delta \frac{(\psi (T)-\psi (0))^{\alpha }}{\Gamma (\alpha +1)}\right) \right\} +  \notag \\
&&L_{f}\left\{ C_{1}C_{\varphi }\varphi (t)+\overline{K}(C_{1}+C_{2})\left(C_{\varphi }C_{\varphi }\varphi (t)+\delta \frac{(\psi (T)-\psi (0))^{\alpha}}{\Gamma (\alpha +1)}\frac{(\psi (t)-\psi (s_{i}))^{\alpha }}{\Gamma(\alpha +1)}\right) \right\}   \notag \\
&\leq &L_{g_{i}}\left\{ C_{2}\delta +L(C_{1}+C_{2})\left( C_{\varphi}\varphi (t)+\delta \frac{(\psi (T)-\psi (0))^{\alpha }}{\Gamma (\alpha +1)}\right) \right\} +  \notag \\
&&L_{f}\left\{ C_{1}C_{\varphi }\varphi (t)+\overline{K}(C_{1}+C_{2})\left(C_{\varphi }^{2}\varphi (t)+\delta \frac{(\psi (T)-\psi (0))^{2\alpha }}{\Gamma (\alpha +1)}\right) \right\}   \notag \\
&\leq &\left[ 
\begin{array}{c}
\left( LC_{\varphi }+L\displaystyle\frac{(\psi (T)-\psi (0))^{\alpha }}{\Gamma (\alpha +1)}+1\right) L_{g_{i}} \\ 
+\left( \overline{K}\displaystyle\frac{(\psi (T)-\psi (0))^{2\alpha }}{\Gamma (\alpha +1)}C_{\varphi }^{2}+\overline{K}C_{\varphi }^{2}+C_{\varphi }\right) L_{f}%
\end{array}%
\right] \left( C_{1}+C_{2}\right) \left( \varphi (t)+\delta \right) .
\end{eqnarray*}

Then, we have 
\begin{equation}
|(\Omega g)(t)-(\Omega h)(t)|\leq \Phi (C_{1}+C_{2})(\varphi (t)+\delta
),\,\,t\in J  \label{*}
\end{equation}%
with 
\begin{equation*}
\Phi :=\left( LC_{\varphi }+L\frac{(\psi (T)-\psi (0))^{\alpha }}{\Gamma(\alpha +1)}+1\right) L_{g_{i}}+\left( \overline{K}\frac{(\psi (T)-\psi
(0))^{2\alpha }}{\Gamma (\alpha +1)}C_{\varphi }^{2}+\overline{K}C_{\varphi}^{2}+C_{\varphi }\right) L_{f}<1.
\end{equation*}

Therefore, from {\rm{Eq.(\ref{*})}}, we conclude that 
\begin{equation*}
{\mbox{d}}(\Omega g,\Omega h)\leq \Phi {\mbox{d}}(g,h)
\end{equation*}
for all $g,h\in X$, provided the condition given by {\rm{Eq.(\ref{2.3})}}.

Now, consider $g_{0}\in X$. From the piecewise continuous property of $g_{0}$ and $\Omega g_{0}$, exists a constant $0<\widetilde{H_{1}}<\infty$, so that
\begin{equation*}
\left\vert (\Omega g_{0})(t)-g_{0}(t)\right\vert =\left\vert 
\begin{array}{c}
\Psi ^{\gamma }\left( t,0\right) x_{0}+\dfrac{1}{\Gamma (\alpha )}\displaystyle\int_{0}^{t}N_{\psi }^{\alpha }(t,s)
f\left( s,g_{0}(s),\displaystyle\frac{1}{\Gamma (\alpha )}\int_{0}^{s}\mathcal{W}_{\psi }^{\alpha }(s,\tau ,x(\tau ))d\tau \right) ds-y_{0}(t)
\end{array}%
\right\vert \leq \widetilde{H_{1}}(\varphi (t)+\delta )
\end{equation*}
for $t\in \lbrack 0,t_{1}]$. On the other hand, also for $0<\widetilde{H_{2}} <\infty $ and $0<\widetilde{H_{3}}<\infty $, 
\begin{equation*}
|(\Omega g_{0})(t)-g_{0}(t)|=\left\vert g_{i}\left( t,g_{0}(t),\frac{1}{\Gamma (\alpha )}\int_{0}^{t}\mathcal{M}_{\psi }^{\alpha }(t,s,g_{0}(s))ds\right)
-g_{0}(t)\right\vert \leq \widetilde{H_{2}}(\varphi (t)+\delta )
\end{equation*}
for any $t\in (t_{i},s_{i}]$ and $i=1,2,...,m$, and 
\begin{equation*}
\left\vert (\Omega g_{0})(t)-g_{0}(t)\right\vert =\left\vert 
\begin{array}{c}
g_{i}\left( s_{i},g_{0}(s_{i}),\displaystyle\frac{1}{\Gamma (\alpha )}\int_{0}^{s_{i}}\mathcal{M}_{\psi }^{\alpha }(t,s,g_{0}(s))ds\right) +\displaystyle\frac{1}{\Gamma (\alpha )}\int_{s_{i}}^{t}N_{\psi }^{\alpha }(t,s)\times  \\ 
f\left( s,g_{0}(s),\displaystyle\frac{1}{\Gamma (\alpha )}\int_{0}^{s}\mathcal{W}_{\psi }^{\alpha }(s,\tau ,g_{0}(\tau ))d\tau \right)
ds-g_{0}(t)%
\end{array}%
\right\vert \leq \widetilde{H_{3}}(\varphi (t)+\delta )
\end{equation*}
for any $t\in (s_{i},t_{i+1}]$ and $i=1,2,...,m$, since $f,g_{i},g_{0},\displaystyle\frac{1}{\Gamma (\alpha )}\int_{0}^{s_{i}}\mathcal{M}_{\psi }^{\alpha }(t,s,g_{0}(s))ds$ and $\displaystyle\frac{1}{\Gamma (\alpha )}\int_{0}^{s}\mathcal{W}_{\psi }^{\alpha }(s,\tau ,g_{0}(\tau ))d\tau $ are bounded on $J$ and $\varphi (\cdot )+\delta >0$.

Then, {\rm{Eq.(\ref{2.5})}} implies that 
\begin{equation*}
{\rm d}(\Omega g_{0},g_{0})<\infty .
\end{equation*}%
Then, there exists a continuous function $y_{0}:J\rightarrow \mathbb{R}$ such that $\Omega ^{n}g_{0}\rightarrow y_{0}$ in $(X,{\rm d})$ as $n\rightarrow \infty $ and $\Omega y_{0}=y_{0}$ that is, $y_{0}$ satisfies {\rm{Eq.(\ref{2.1})}} for every $t\in J$ by means of Banach fixed point theorem {\rm (Theorem \ref{teo1})}.

Now, the next step, we will prove that 
\begin{equation*}
X=\{g\in X/{\rm d}(g_{0},g)<\infty \}.
\end{equation*}%

\textbf{Let} $g\in X$. Note that $g$ and $g_{0}$ are bounded on $J$ and $\underset{t\in J}{\min }(\varphi (t)+\delta )>0$, then there exists a constant $0<C_{g}<\infty $ such that 
\begin{equation*}
|g_{0}(t)-g(t)|\leq C_{g}(\varphi (t)+\delta )
\end{equation*}
for any $t\in J$. Hence, we have ${\rm d}(g_{0},g)<\infty $ for all $g\in X$ that is 
\begin{equation*}
X=\{g\in X/{\rm d}(g_{0},g)<\infty \}.
\end{equation*}

Finally, we conclude the first part, i.e., $y_{0}$ is the unique continuous function satisfying {\rm{Eq.(\ref{2.1})}}.

On the other hand, by means of {\rm{Eq.(\ref{4})}} and of the hypothesis $H_{6}$, we have
\begin{eqnarray}\label{2.8}
{\rm d}(y,\Omega y) &=&\left\vert 
\begin{array}{c}
y(t)-\Psi ^{\gamma }\left( t,0\right) x_{0} -\displaystyle\frac{1}{\Gamma (\alpha )}\int_{0}^{t}N_{\psi }^{\alpha }(t,s)f\left( s,y(s),\displaystyle\frac{1}{\Gamma (\gamma )}\int_{0}^{s}\mathcal{W}_{\psi }^{\alpha }(s,\tau ,y(\tau )){\rm d}\tau \right) ds\notag
\end{array}%
\right\vert  \\
&\leq &\displaystyle\frac{1}{\Gamma (\alpha )}\int_{0}^{t}N_{\psi }^{\alpha}(t,s)\varphi (s)ds\leq C_{\varphi }\varphi (t)\leq 1+C_{\varphi }.
\end{eqnarray}

Multiplying both sides of {\rm{Eq.(\ref{2.8})}} by $1-\Phi $, we have 
\begin{equation*}
{\rm d}(y,y_{0})\leq \frac{{\rm d}(\Omega y,y)}{1-\Phi }\leq \frac{1+C_{\varphi }}{1-\Phi },\,\,t\in J.
\end{equation*}

By this last expression, we conclude that {\rm{Eq.(\ref{2.2})}} is true.
\end{proof}
%%%%%%%%%%%%%%%%%%%%%%%%%%%%%%%%%%%%%%%%%%%%%%%%%%%%%%%%%%%%%%%%%%%%%%%%%%%%%%%%%%%%%%%%%%%
%%%%%%%%%%%%%%%%%%%%%%%%%%%%%%%%%%%%%%%%%%%%%%%%%%%%%%%%%%%%%%%%%%%%%%%%%%%%%%%%%%%%%%%%%%%
\section{Examples}
In this section we will present only two examples, considered as particular cases of fractional integrodifferential equations, specifically one of 
them \textbf{is} the Riemann-Liouville sense (fractional) and another relative to the integer order derivative.  We \textbf{recall} that, our result, 
involving the $\psi$-Hilfer fractional derivative, is the general case, because both Riemann-Liouville and integer order are recovered as particular cases.

First, we consider
\begin{equation}\label{p1}
\left\{ 
\begin{array}{ccc}
^{H}{\mathbb{D}}_{s_{i},0^{+}}^{\alpha ,\beta ,\psi }x(t) & = & f\left( t,x(t),\displaystyle\frac{1}{\Gamma (\alpha )}\displaystyle\int_{0}^{t}\mathcal{W}_{\psi }^{\alpha }(t,s,x(s))ds\right)
,\,\,t\in \lbrack 0,1] \\ 
x(t) & = & g_{i}\left( t,x(t),\displaystyle\frac{1}{\Gamma (\alpha )}\displaystyle\int_{0}^{t}\mathcal{M}_{\psi}^{\alpha }(t,s,x(s))ds\right) ,\,\,t\in (1,2]
\end{array}%
\right. 
\end{equation}
and
\begin{equation}\label{p2}
\left\{ 
\begin{array}{ccc}
\left\vert {}^{H}{\mathbb{D}}_{s_{i},0^{+}}^{\alpha ,\beta ,\psi }y(t)-f\left(t,y(t),\displaystyle\frac{1}{\Gamma (\alpha )}\displaystyle\int_{0}^{t}\mathcal{W}_{\psi }^{\alpha}k(t,s,y(s))ds\right) \right\vert  & \leq  & \mathbb{E}_{\alpha
}(t),\,\,t\in \lbrack 0,1] \\ 
\left\vert y(t)-g_{i}\left( t,y(t),\displaystyle\frac{1}{\Gamma (\alpha )}\displaystyle\int_{0}^{t}\mathcal{M}_{\psi }^{\alpha }(t,s,y(s))ds\right) \right\vert  & \leq  & 1,\,\,t\in (1,2]
\end{array}%
\right. 
\end{equation}
with 
\begin{eqnarray}
f\left( t,x(t),\frac{1}{\Gamma (\alpha )}\int_{0}^{t}\mathcal{W}_{\psi }^{\alpha}(t,s,x(s))ds\right) 
&=&\frac{1}{5+\psi (t)}\left( |x(t)|+\frac{1}{\Gamma (\alpha )}\int_{0}^{t} \psi'(s)(\psi(t)-\psi(s))^{\alpha-1}  \frac{|x(s)|}{10+\psi (s)}ds\right),\,\,t\in \lbrack 0,1]  \notag
\end{eqnarray}%
and 
\begin{equation}
x(t)=\frac{1}{(5+\psi (t))(1+|x(t)|)}\left( |x(t)|\int_{0}^{t}\psi'(s)(\psi(t)-\psi(s))^{\alpha-1}\frac{|x(s)|}{15+\psi (s)}ds\right) ,\,\,t\in (1,2]
\end{equation}
where $\mathbb{E}_{\alpha}(\cdot)=\overset{\infty }{\underset{k=0}{\sum }}\dfrac{t^{k}}{\Gamma \left( \gamma k+1\right) }$ is the one-parameter Mittag-Leffler 
function with $0 < \alpha \leq 1$. Also, relatively to Eq.(\ref{p2}) we simply exchange $x$ for $y$.

\subsection{Riemann-Liouville fractional derivative}

In this case we consider Eq.(\ref{p1}) and Eq.(\ref{p2}) with $J=[0,2]$; $0=t_0=s_0<t_1=1<s_1 = 2$. Taking $\psi(t)=t$, $\alpha = 1/2$, the limit $\beta \to 0$ and 
the following nonlinear functions
\begin{equation*}
f\left( t,x(t),\frac{1}{\Gamma (1/2)}\int_{0}^{t}\mathcal{W}_{t}^{\frac{1}{2}}(t,s,x(s))ds\right) =\frac{1}{5+t}
\left( |x(t)|+\frac{1}{\Gamma (1/2)}\int_{0}^{t}(t-s)^{-\frac{1}{2}}\frac{|x(s)|}{10+s}ds\right) 
\end{equation*}%
with $k(t,x(t))=\displaystyle\frac{|x(t)|}{10+t}$, $\overline{k}=1/10$, $L_{f}=1/5$, $t\in \lbrack 0,1]$ and 
\begin{equation*}
g_{i}\left( t,x(t),\frac{1}{\Gamma (1/2)}\int_{0}^{t}\mathcal{M}_{t}^{\frac{1}{2}}(t,s,x(s))ds\right) =\frac{1}{(5+t)(1+|x(t)|)}\left( |x(t)|+
\frac{1}{\Gamma(1/2)}\int_{0}^{t}(t-s)^{-\frac{1}{2}}\frac{|x(s)|}{15+s}ds\right) 
\end{equation*}%
with $\ell (t,x(t))=\displaystyle\frac{|x(t)|}{15+t}$, $L=1/15$, $L_{g_{i}}=1/5$, $t\in (1,2]$ we get a particular case of Eq.(\ref{p1}) and Eq.(\ref{p2}), 
involving Riemann-Liouville fractional derivative.

We put $\varphi(t) = \mathbb{E}_{\alpha}(t)$ and $\delta = 1$. Then, for $C_{\varphi}=1$ we have 
\begin{equation*}
\frac{1}{\Gamma (1/2)}\int_{0}^{t}N_{t}^{\frac{1}{2}}(t,s)\mathbb{E}_{1/2}(s)ds\leq \mathbb{E}_{1/2}(t).
\end{equation*}

\textbf{By,} a simple but tedious calculation, we get $\Phi = \dfrac{3}{8} < 1$ and by Theorem \ref{teo2.1} there exists a unique solution $y_0: [0,2] \to \mathbb{R}$ such that 
\begin{equation*}
y_{0}(t)=\frac{t^{-1/2}}{\Gamma (1/2)}x(0)+\frac{1}{\Gamma (1/2)}\int_{0}^{t}(t-s)^{-1/2}\frac{1}{5+s}\left( |y_{0}(s)|+
\frac{1}{\Gamma (1/2)}\int_{0}^{s}(s-\tau )^{-1/2}\frac{y_{0}(\tau )}{10+\tau }d\tau \right) ds, \text{ } \mbox{t $\in [0,1]$},
\end{equation*}
and 
\begin{equation}
y_{0}(t)=\frac{1}{(5+t)(1+|y_{0}(t)|)}\left( |y_{0}(t)|+\frac{1}{\Gamma (1/2)}\int_{0}^{t}(t-s)^{-1/2}\frac{y_{0}(s)}{15+s}ds\right), \text{ } \mbox{t $\in (1,2]$}, 
\end{equation}%
and 
\begin{equation}
|y(t)-y_{0}(t)|\leq \frac{(1+C_{\varphi })(\varphi (t)+\delta )}{1-\Phi }=\frac{54}{19}\mathbb{E}_{1/2}(t+1),\,\,t\in \lbrack 0,2].
\end{equation}

\subsection{Integer derivative}
As in the precedent case, consider Eq.(\ref{p1}) and Eq.(\ref{p2}) with $J=[0,2]$; $0=t_0=s_0<t_1=1<s_1 = 2$. Taking $\psi(t)=t$, $\alpha = 1$, $\beta =1/2$ and the 
following nonlinear functions
\begin{equation*}
f\left( t,x(t),\int_{0}^{t}\mathcal{W}_{t}^{1}(t,s,x(s))ds\right) =\frac{1}{5+t}\left(|x(t)|+\int_{0}^{t}N_{t}^{1}(t,s)\frac{|x(s)|}{10+s}ds\right) 
\end{equation*}
with $k(t,x(t))=\displaystyle\frac{|x(t)|}{10+t}$, $\overline{k}=1/10$, $L_{f}=1/5$, $t\in \lbrack 0,1]$ and 
\begin{equation*}
g_{i}\left( t,x(t),\int_{0}^{t}\mathcal{M}_{t}^{1}(t,s,x(s))ds\right) =\frac{1}{%
	(5+t)(1+|x(t)|)}\left( |x(t)|+\int_{0}^{t}N_{t}^{1}(t,s)\frac{|x(s)|}{15+s}%
ds\right) 
\end{equation*}
with $\ell (t,x(t))=\displaystyle\frac{|x(t)|}{15+t}$, $L=1/15$, $ L_{g_{i}}=1/5$, $t\in (1,2]$ we get a particular case of Eq.(\ref{p1}) and Eq.(\ref{p2}), 
involving integer derivative. We put $\varphi (t)=\mathbb{E}_{1}(t)=e^{t}$ and $\delta =1$. Then, for $C_{\varphi }=1$ we have 
\begin{equation*}
\int_{0}^{t}N_{t}^{1}(t,s)\mathbb{E}_{1}(s)ds\leq \mathbb{E}_{1}(t)=e^{t}.
\end{equation*}

Also, here, after a simple and tedious calculation we can show $\Phi =\dfrac{14}{25} < 1$ and by Theorem \ref{teo2.1}, there exists a unique solution $y_0:[0,2] \to \mathbb{R}$ such that
\begin{equation*}
y_{0}(t)=x(0)+\int_{0}^{t}\frac{1}{5+s}\left( |y_{0}(s)|+\int_{0}^{s}\frac{|y_{0}(\tau )|}{1-+\tau }d\tau \right) ds,\,\,t\in \lbrack 0,1]
\end{equation*}
and 
\begin{equation*}
y_{0}(t)=\frac{1}{(5+t)(1+|y_{0}(t)|)}\left( |y_{0}(t)|+\int_{0}^{t}\frac{|y_{0}(s)|}{15+s}ds\right) ,\,\,t\in (1,2]
\end{equation*}
and 
\begin{equation*}
|y(t)-y_{0}(t)|\leq \frac{(1+C_{\varphi })(\varphi (t)+\delta )}{1-\Phi }=\frac{50}{11}(e^{t}+1),\,\,t\in \lbrack 0,2].
\end{equation*}

%%%%%%%%%%%%%%%%%%%%%%%%%%%%%%%%%%%%%%%%%%%%%%%%%%%%%%%%%%%%%%%%%%%%%%%%%%%%%%%%%%%%%%%%%%%%%%%%%%%%%%%%%%

\section{Concluding Remarks}

We can conclude that the main results of this paper have been successfully
achieved, through Banach's fixed point theorem. We investigated the
existence and Ulam--Hyers--Rassias stability of an impulsive
integrodifferential equation written with the $\psi$-Hilfer fractional
derivative. We hope that this paper will prove to be important to
mathematics, especially for researchers working with problems
involving impulsive fractional differential equations. 

A possible continuation of this research might be finding an answer to this
question: Would it be possible to obtain results similar to the ones we have presented in 
space $L_{p,\alpha}([0,1],\mathbb{R})$ with norm $\Vert (\cdot)
\Vert_{p,\alpha}$ using the $\psi$-Hilfer fractional derivative, Brouwer's fixed
point and/or Schauder's fixed point theorems\cite{wang2017,book1,book2}? 

\section*{Acknowledgment}
We are grateful to Dr. J. Em\'{\i}lio Maiorino for several fruitful discussions and to
the anonymous referees, whose suggestions improved the paper.

\bibliography{ref}
\bibliographystyle{plain}

\end{document}